\documentclass[11pt]{article}
\usepackage[margin=1.0in]{geometry}

\usepackage{hyperref}
\usepackage{amsmath}
\usepackage{amsthm}
\usepackage{amssymb}
\usepackage{amsfonts}
\usepackage{multicol}
\usepackage{subfigure}
\usepackage[numbers,sort&compress]{natbib}
\usepackage{booktabs}
\usepackage{titlesec}
\usepackage{longtable,tabu}

\usepackage{enumitem}
\setlist[enumerate]{label={\arabic*.}}

\usepackage{color}

\usepackage{tikz}
\usetikzlibrary{calc}

\usepackage{tkz-graph}

\newtheorem{theorem}{Theorem}[section]
\newtheorem{lemma}[theorem]{Lemma}

\newtheorem{corollary}[theorem]{Corollary}

\theoremstyle{definition}

\newcommand{\squishlist}{
 \begin{list}{$\bullet$}
  { \setlength{\itemsep}{0pt}
     \setlength{\parsep}{3pt}
     \setlength{\topsep}{3pt}
     \setlength{\partopsep}{0pt}
     \setlength{\leftmargin}{2.5em}
     \setlength{\labelwidth}{1em}
     \setlength{\labelsep}{0.5em} } }

\newcommand{\squishlisttwo}{
 \begin{list}{$\triangleright$}
  { \setlength{\itemsep}{0pt}
     \setlength{\parsep}{0pt}
    \setlength{\topsep}{0pt}
    \setlength{\partopsep}{0pt}
    \setlength{\leftmargin}{2em}
    \setlength{\labelwidth}{1.5em}
    \setlength{\labelsep}{0.5em} } }

\newcommand{\squishend}{
  \end{list}  }
  
\newcommand{\kcol}{$k$-\textsc{Colouring}}

\newcommand{\squid}[1]{$(4,#1)$-$squid$}
\newcommand{\hl}[1]{$(3,#1)$-$squid$}
\newcommand{\noneighbs}{\overline{N[v]}}

\begin{document}

\title{On the finiteness of $k$-vertex-critical $2P_2$-free graphs with forbidden induced squids or bulls}
\author{
Melvin Adekanye\\
\small Department of Computing Science\\
\small The King's University\\
\small Edmonton, AB Canada\\
\and
Christopher Bury\\ 
\small Department of Computing Science\\
\small The King's University\\
\small Edmonton, AB Canada\\
\and
Ben Cameron\\ 
\small School of Mathematical and Computational Sciences\\
\small University of Prince Edward Island\\
\small Charlottetown, PE Canada\\
\small brcameron@upei.ca\\
\and
Thaler Knodel\\ 
\small Department of Computing Science\\
\small The King's University\\
\small Edmonton, AB Canada
}

\date{\today}

\maketitle

\begin{abstract}
A graph is $k$-vertex-critical if $\chi(G)=k$ but $\chi(G-v)<k$ for all $v\in V(G)$ and $(G,H)$-free if it contains no induced subgraph isomorphic to $G$ or $H$. We show that there are only finitely many $k$-vertex-critical $(2P_2,H)$-free graphs for all $k$ when $H$ is isomorphic to any of the following graphs of order $5$:
\begin{itemize}
\item $bull$,
\item $chair$,
\item $claw+P_1$, or 
\item $\overline{diamond+P_1}$.
\end{itemize}

\noindent The latter three are corollaries of more general results where $H$ is isomorphic to $(m, \ell)$-$squid$ for $m=3,4$ and any $\ell\ge 1$ where an $(m,\ell)$-$squid$ is the graph obtained from an $m$-cycle by attaching $\ell$ leaves to a single vertex of the cycle. For each of the graphs $H$ above and any fixed $k$, our results imply the existence of polynomial-time certifying algorithms for deciding the $k$-colourability problem for $(2P_2,H)$-free graphs. Further, our structural classifications allow us to exhaustively generate, with aid of computer search, all $k$-vertex-critical $(2P_2,H)$-free graphs for $k\le 7$ when $H=bull$ or $H=(4,1)$-$squid$ (also known as $banner$).
\end{abstract}


\section{Introduction}

The \kcol{} decision problem is to determine, for fixed $k$, if a given graph admits a proper $k$-colouring. This problem is of considerable interest in computational complexity theory since for all $k\ge 3$, it is one of the most intuitive of Karp's~\cite{Karp1972} original 21 NP-complete problems. Research has focused on many computational aspects of \kcol{} including approximation~\cite{approximategraphcoloring} and heuristic algorithms~\cite{heuristicgraphcoloring}, but we are most interested in the substructures that can be forbidden to produce polynomial-time algorithms to solve \kcol{} for all $k$. The substructures we are interested in forbidding are induced subgraphs. We say a graph is $H$-free if it contains no induced copy of $H$. One of the most impressive results to this end is Ho\`{a}ng et al.'s 2010 result that \kcol{} can be solved in polynomial-time on $P_5$-free graphs for all $k$~\cite{Hoang2010}. The full strength of this result was later demonstrated by Huang's~\cite{Huang2016} result that \kcol{} remains NP-complete for $P_6$-free graphs when $k\ge 5$, and for $P_7$-free graphs when $k\ge 4$. A polynomial-time algorithm to solve $4$-\textsc{Colouring} for $P_6$-free graphs was later developed by Chudnovsky et al.~\cite{P6free1,P6free2,P6freeconf}. It has long been known that \kcol{} remains NP-complete on $H$-free graphs when $H$ contains an induced cycle~\cite{KaminskiLozin2007,MaffrayMorel2012} or claw~\cite{LevenGail1983,Holyer1981}. Thus, $P_5$ is the largest connected subgraph that can be forbidden for which \kcol{} can be solved in polynomial-time for all $k$ (assuming P$\neq$NP). As such, deeper study of $P_5$-free graphs in relation to \kcol{} has attracted significant attention. One of these deeper areas is determining which subfamilies of $P_5$-free graphs admit polynomial-time \kcol{} algorithms that are fully \textit{certifying}. 

An algorithm is certifying if it returns with each output, an easily verifiable witness (called a \textit{certificate}) that the output is correct. For \kcol{}, a certificate for ``yes'' output is a $k$-colouring of the graph, and indeed the \kcol{} algorithms for $P_5$-free graphs from~\cite{Hoang2010} return $k$-colourings if they exist. To discuss certificates for ``no'' output for \kcol{}, we must first define that a graph is \textit{$k$-vertex-critical} if it is not $(k-1)$-colourable, but every proper induced subgraph is. Since every graph that is not $k$-colourable must contain an induced $(k+1)$-vertex-critical graph, such an induced subgraph is a certificate for ``no'' output of \kcol{}. Indeed, if there are only finitely many $(k+1)$-vertex-critical graphs in a family $\mathcal{F}$, then a polynomial-time algorithm to decide \kcol{} that certifies ``no'' output for all graphs in $\mathcal{F}$ can be implemented by searching for each $(k+1)$-vertex-critical graph in $\mathcal{F}$ as an induced subgraph of the input graph (see~\cite{P5banner2019}, for example, for the details). Thus, proving there are only finitely many $(k+1)$-vertex critical graphs in a subfamily of $P_5$-free graphs implies the existence of polynomial-time certifying \kcol{} algorithms when paired with the algorithms that certify ``yes'' output from~\cite{Hoang2010}. 

It should also be noted that classifying vertex-critical graphs is of interest in resolving the Borodin–Kostochka Conjecture~\cite{BKconj1977} that states if $G$ is a graph with $\Delta(G)\ge 9$ and $\omega(G)\le \Delta(G)-1$, then $\chi(G)\le \Delta(G)-1$. This connection comes in the proof technique often employed (see for example \cite{CranstonLafayetteRabern2022,HaxellNaserasr2023,WuWu023}) to consider a minimal counterexample to the conjecture which must be vertex-critical~\cite{GuptaPradhan2021}.

A natural question to ask then is for which $k$ are there only finitely many $k$-vertex-critical $P_5$-free graphs. Bruce et al.~\cite{Bruce2009} showed that there are only finitely many $4$-vertex-critical $P_5$-free graphs, and this result was later extended by Maffray and Morel~\cite{MaffrayMorel2012} to develop a linear-time certifying algorithm for $3$-\textsc{Colouring} $P_5$-free graphs. Unfortunately, for each $k\ge 5$, Ho\`{a}ng et al.~\cite{Hoang2015} constructed infinitely many $k$-vertex-critical $P_5$-free graphs. This result caused attention to shift to subfamilies of $P_5$-free graphs, usually defined by forbidding additional induced subgraphs. A graph is $(H_1,H_2,\dots,H_{\ell})$-free if it does not contain an induced copy of $H_i$ for all $i\in\{1,2,\dots,\ell\}$. It is known that there are only finitely many $5$-vertex-critical $(P_5,H)$-free graphs if $H$ is isomorphic to $C_5$~\cite{Hoang2015}, $bull$~\cite{HuangLiXia2023}, or $chair$~\cite{HuangLi2023} (where $bull$ is the graph in Figure~\ref{fig:bull} and $chair$ is the subgraph of the graph in Figure~\ref{subfig:4squid} induced by the set $\{u_1,u_2,u_4,w_1,w_2\}$). When moving beyond $k=5$ to larger values of $k$, the most comprehensive result is K. Cameron et al.'s~\cite{KCameron2021} dichotomy theorem that there are infinitely many $k$-vertex-critical $(P_5,H)$-free graphs for $H$ of order $4$ if and only if $H$ is $2P_2$ or $K_3+P_1$. An open question was also posed in~\cite{KCameron2021} to prove a dichotomy theorem for $H$ of order $5$, and there has been substantial work toward this end. On the positive side of resolving this open question, it is known that there are only finitely many $k$-vertex-critical $(P_5,H)$-free graphs for all $k$ if $H$ is isomorphic to $banner$~\cite{Brause2022}, $dart$~\cite{Xiaetal2023}, $K_{2,3}$~\cite{Kaminski2019}, $\overline{P_5}$~\cite{Dhaliwal2017}, $P_2+3P_1$~\cite{CameronHoangSawada2022}, $P_3+2P_1$~\cite{AbuadasCameronHoangSawada2022}, $gem$~\cite{CaiGoedgebeurHuang2021,CameronHoang2023}, or $\overline{P_3+P_2}$~\cite{CaiGoedgebeurHuang2021} (where we refer to the table in~\cite{CameronHoang2023P5C5} for the adjacencies of each of these graphs). On the negative side of resolving the open question is that there are infinitely many $k$-vertex-critical $(P_5,H)$-free graphs for every graph $H$ of order $5$ containing an induced $2P_2$ or $K_3+P_1$ (a corollary from the dichotomy theorem in~\cite{KCameron2021}) for $k\ge 5$, and for $H=C_5$ for all $k\ge 6$~\cite{CameronHoang2023P5C5}. It therefore remains unknown for which graphs $H$ of order $5$ there are only finitely many $(P_5,H)$-free graphs for all $k$ when $H$ is one of the following:

\begin{multicols}{3}
\begin{itemize}
 \item $claw+P_1$
 \item $P_4+P_1$ 
  \item $chair$
 \item $\overline{diamond+P_1}$ 
 \item $C_4+P_1$
 \item $bull$ 
 \item $\overline{K_3+2P_1}$
 \item  $\overline{P_3+2P_1}$
 \item $W_4$
 \item $K_5-e$
 \item $K_5$
\end{itemize}
\end{multicols}

In this paper, we prove that there are only finitely many $k$-vertex-critical $(2P_2,H)$-free graphs for all $k$ for four of the graphs above. Namely,  $claw+P_1$,  $\overline{diamond+P_1}$, $chair$, and $bull$. Three of these are corollaries of more general results that there are only finitely many $k$-vertex-critical ($2P_2$, \squid{\ell})-free graphs and ($2P_2$,\hl{\ell})-free graphs for all $k,\ell$ (see Figure~\ref{fig:mellsquid}). We also show that there are only finitely many $(2P_2,K_3+P_1,P_4+P_1)$-free graphs for all $k$. These results, while not as strong as they would be if they were for $P_5$-free graphs instead of $2P_2$-free graphs, are nonetheless important as every infinite family of $k$-vertex-critical $P_5$-free graphs known (i.e., the families from~\cite{Hoang2015,Chudnovsky4criticalconnected2020} and the generalized families from~\cite{CameronHoang2023P5C5}) are actually $(2P_2,K_3+P_1)$-free as well. Thus, there is no known $H$ for which there are infinitely many $k$-vertex-critical $(P_5,H)$-free graphs and only finitely many that are $(2P_2,H)$-free for some $k$. Therefore, our results provide strong evidence for the finiteness of $k$-vertex-critical $(P_5,H)$-free graphs for each $H$ we consider, but will also still be of interest if it turns out that there are infinitely many that are $(P_5,H)$-free but only finitely many that are $(2P_2, H)$-free.

In addition, while most proof techniques for showing finiteness of vertex-critical graphs involve bounding structure around odd holes or anti-holes using Strong Perfect Graph Theorem~\cite{Chudnovsky2006}, or using a clever application of Ramsey's Theorem, our techniques are novel and simple. We show that $k$-vertex-critical graphs in each of the families in question are necessarily $(P_3+m P_1)$-free for some $m$ depending only on $k$ and the orders of the forbidden induced subgraphs.  We then apply the result from~\cite{AbuadasCameronHoangSawada2022} that there are only finitely many $k$-vertex-critical $(P_3+m P_1)$-free graphs for all $m$ and $k$ to get the immediate corollary of finiteness.  This approach results in short and easy-to-follow proofs that only require basic facts about vertex-critical graphs to prove finiteness. 

\subsection{Outline}
The rest of this paper is structured as follows. We include the background and preliminary results that are used in the proofs of our main results in Section~\ref{sec:prelims}. We prove that there are only finitely many $k$-vertex-critical $(2P_2,bull)$-free graphs for all $k$ in Section~\ref{sec:bull}. We then prove two more general results which have as corollaries that there are only finitely many $k$-vertex-critical $(2P_2,H)$-free graphs for all $k$ when $H$ isomorphic to $chair$, $claw+P_1$ (Section~\ref{sec:ellsquid}), and $\overline{diamond+P_1}$ (Section~\ref{sec:fountain}). In Section~\ref{sec:critgen}, we discuss methods for using a computer-aided search to exhaustively generate all $k$-vertex-critical graphs in certain families and include our enumerations for the complete sets of all such graphs that are $(2P_2,H)$-free graphs for $H=bull$ or $H=banner$ for all $k\le 7$. Finally, we conclude the paper with a discussion on future research directions and we conclude this section with a brief subsection outlining definitions and notation. 

\subsection{Notation}

If vertices $u$ and $v$ are adjacent in a graph we write $u\sim v$ and if they are nonadjacent we write $u\nsim v$. For a vertex $v$ in a graph, $N(v)$, $N[v]$ and $\noneighbs$ denote the open neighbourhood, closed neighbourhood, and set of nonneighbours of $v$, respectively. More precisely, $N(v)=\{u\in V(G): u\sim v\}$, $N[v]=N(v)\cup\{v\}$, and $\noneighbs=V(G)\setminus N[v]$. We call a vertex $v$ of a graph $G$ \textit{universal} if $N[v]=V(G)$, and \textit{nonuniversal} otherwise. We let $\delta(G)$ and $\Delta(G)$ denote the minimum and maximum degrees of $G$, respectively. We let $\alpha(G)$ denote the independence number of $G$. For subsets $A$ and $B$ of $V(G)$, we say $A$ is \textit{(anti)complete} to $B$ if $a$ is (non)adjacent to $b$ for all $a\in A$ and $b\in B$. If $A=\{a\}$ then we simplify notation and say $a$ is (anti)complete to $B$. . We use $\chi(G)$ and $\omega(G)$ to denote the \textit{chromatic number} and \textit{clique number} of $G$, respectively.

\section{Preliminaries}\label{sec:prelims}

We will make extensive use of the following lemma and theorem throughout the paper.

\begin{lemma}[\cite{Hoang2015}]\label{lem:nocomparablecliques}
Let $G$ be a graph with chromatic number $k$. If G contains two disjoint $m$-cliques $A = \{a_1, a_2,\ldots , a_m\}$ and $B = \{b_1, b_2,\ldots , b_m\}$ such that $N(a_i) \setminus A \subseteq N(b_i) \setminus B$ for all $1 \le i \le m$, then $G$ is not $k$-vertex-critical.
\end{lemma}

\noindent We stated Lemma~\ref{lem:nocomparablecliques} here in its full generality for interested readers, but we will only use it for the case where $m=1$. For easier reference in this case, we call vertices $a$ and $b$ \textit{comparable} if $N(a)\subseteq N(b)$. The contrapositive of Lemma~\ref{lem:nocomparablecliques} for $m=1$, then can be restated as there are no comparable vertices in a vertex-critical graph.

\begin{theorem}[\cite{CameronHoangSawada2022}]\label{thm:finiteP3ellP1freecrit}
There are only finitely many $k$-vertex-critical $(P_3+\ell P_1)$-free graphs for all $k\ge 1$ and $\ell \ge 0$.
\end{theorem}

We also require two new technical lemmas that will be used multiple times throughout the paper as we apply Theorem~\ref{thm:finiteP3ellP1freecrit} to prove the finiteness of vertex-critical graphs in many families.

\begin{lemma}\label{lem:2P2freenonneighbconnected}
If $G$ is a $k$-vertex-critical $2P_2$-free graph, then for every nonuniversal vertex $v\in V(G)$,  $\noneighbs$ induces a connected graph with at least two vertices.
\end{lemma}
\begin{proof}
Let $G$ be a $k$-vertex-critical $2P_2$-free graph, $v\in V(G)$ be nonuniversal, and $H$ be the graph induced by $\noneighbs$. If $u\in \noneighbs$ such that $u$ is an isolated vertex in the graph induced by $H$, then $N(u)\subseteq N(v)$ contradicting $G$ being $k$-vertex-critical by Lemma~\ref{lem:nocomparablecliques}. Therefore, if $H$ has at least two components, then each component has at least one edge and therefore taking an edge from each component induces a $2P_2$. This contradicts $G$ being $2P_2$-free. 
\end{proof}

\begin{lemma}\label{lem:2P2freeneighboursofS}
Let $G$ be a $2P_2$-free graph that contains an induced $P_3+\ell P_1$ for some $\ell\ge 1$. Let $S\cup\{v_1,v_2,v_3\}\subseteq V(G)$ induce a $P_3+\ell P_1$ where $v_1v_2v_3$ is the induced $P_3$ and $S$ contains the vertices in the $\ell P_1$. Then for ever vertex $u\in \overline{N[v_2]}$ such that $u$ has a neighbour in $S$, $u$ is complete to $\{v_1,v_3\}$. 
\end{lemma}
\begin{proof}
Let $s\in S$ and $u\in \overline{N[v_2]}$ with $u\sim s$. Let $i\in \{1,3\}$. If $u\nsim v_i$, then $\{s,u,v_i,v_2\}$ induces a $2P_2$ in $G$, a contradiction. Thus, $u$ is complete to $\{v_1,v_3\}$.
\end{proof}



Finally, there are special families of graphs that will be used in our results. For $\ell\ge 1$ and $m\ge 3$, let the $(m,\ell)$-$squid$ be the graph obtained from $C_m$ by attaching $\ell$ leaves to one of its vertices. Figure~\ref{fig:mellsquid} shows these graphs for $m=3,4$, which are the ones of interest to us, since for $m>4$ $(m,\ell)$-$squid$ is not $2P_2$-free.
\begin{center}
\begin{figure}[h]
\centering
\qquad
\subfigure[$m=3$.]{
\def\r{1.5}
\begin{tikzpicture}
\begin{scope}[every node/.style={circle,fill,draw}]
    \node[label=above:$u_2$,draw] (u2) at (-1*\r,0*\r) {};
    \node[label=above:$u_3$,draw] (u3) at (1*\r,0*\r) {};
    \node[label=above:$u_1$,draw] (u4) at (0*\r,-1*\r) {};
    \node[label=below:$w_1$,draw] (w1) at (-1*\r,-2*\r) {};
    \node[label=below:$w_2$,draw] (w2) at (-0.5*\r,-2*\r) {};  
    \node[label=below:$w_{\ell}$,draw] (well) at (1*\r,-2*\r) {};    
\end{scope}

\begin{scope}

    \path [-] (u4) edge node {} (w1);
    \path [-] (u4) edge node {} (w2);
    \path [-] (u4) edge node {} (well);
    
    \path [-] (u4) edge node {} (u2);
    \path [-] (u4) edge node {} (u3);
    
    \path [-] (u2) edge node {} (u3);

\end{scope}

\path (w2) -- node[auto=false]{\ldots} (well);
\end{tikzpicture}
\label{subfig:3squid}
}
\qquad
\subfigure[$m=4$.]{\def\r{1.5}
\centering
\begin{tikzpicture}
\begin{scope}[every node/.style={circle,fill,draw}]
    \node[label=above:$u_1$,draw] (u1) at (0*\r,1*\r) {};
    \node[label=above:$u_2$,draw] (u2) at (-1*\r,0*\r) {};
    \node[label=above:$u_3$,draw] (u3) at (1*\r,0*\r) {};
    \node[label=above:$u_4$,draw] (u4) at (0*\r,-1*\r) {};
    \node[label=below:$w_1$,draw] (w1) at (-1*\r,-2*\r) {};
    \node[label=below:$w_2$,draw] (w2) at (-0.5*\r,-2*\r) {};  
    \node[label=below:$w_{\ell}$,draw] (well) at (1*\r,-2*\r) {};    
\end{scope}

\begin{scope}

    \path [-] (u4) edge node {} (w1);
    \path [-] (u4) edge node {} (w2);
    \path [-] (u4) edge node {} (well);
    
    \path [-] (u4) edge node {} (u2);
    \path [-] (u4) edge node {} (u3);
    
    \path [-] (u1) edge node {} (u2);
    \path [-] (u1) edge node {} (u3);
    
\end{scope}

\path (w2) -- node[auto=false]{\ldots} (well);

\end{tikzpicture}
\label{subfig:4squid}
}
\caption{The general form of the $(m,\ell)$-$squid$ graphs for $m=3$ and $m=4$.}\label{fig:mellsquid}
\end{figure}
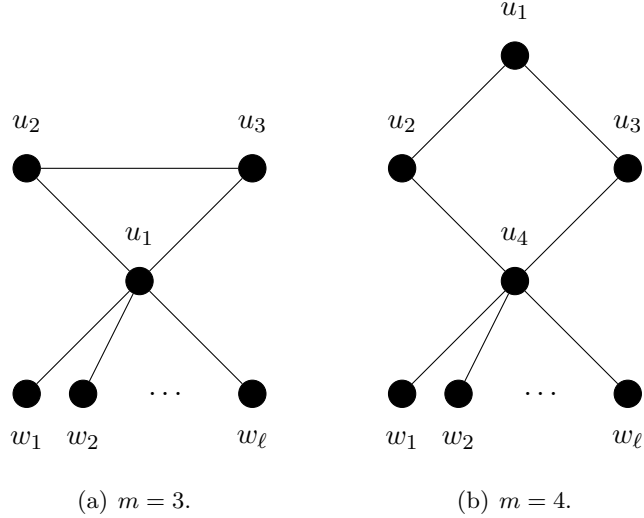
\end{center}

\section{($2P_2$, $bull$)-free}\label{sec:bull}
Let the $bull$ be the graph obtained from a $K_3$ by attaching two leaves, each to a different vertex of the $K_3$ and shown in Figure~\ref{fig:bull}.

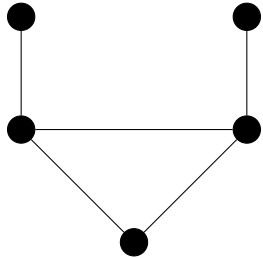
\begin{figure}[h]
\def\r{1.5}
\centering
\begin{tikzpicture}
\begin{scope}[every node/.style={circle,fill,draw}]
    \node (u1) at (-1*\r,1*\r) {};
    \node (u2) at (-1*\r,0*\r) {};
    \node (u3) at (1*\r,0*\r) {};
    \node (u4) at (0*\r,-1*\r) {};
    \node (w1) at (1*\r,1*\r) {};   
\end{scope}

\begin{scope}

    \path [-] (u4) edge node {} (u2);
    \path [-] (u4) edge node {} (u3);
    \path [-] (u2) edge node {} (u3);
    
    \path [-] (u1) edge node {} (u2);
    \path [-] (w1) edge node {} (u3);
    
\end{scope}

\end{tikzpicture}
\caption{The $bull$ graph.}\label{fig:bull}
\end{figure}

\begin{lemma}\label{lem:2P2bullP3P1fee}
Let $k\ge 1$. If $G$ is $k$-vertex-critical $(2P_2,bull)$-free, then $G$ is $(P_3+ P_1)$-free.
\end{lemma}
\begin{proof}
Let $G$ be a $k$-vertex-critical $(2P_2,bull)$-free graph and by way of contradiction let $\{v_1,v_2,v_3,s_1\}$ induce a $P_3+P_1$ in $G$ where $\{v_1,v_2,v_3\}$ induces the $P_3$, in that order, of the $P_3+P_1$. We will show that $v_1$ and $v_3$ are comparable which will contradict Lemma~\ref{lem:nocomparablecliques}.

We first show that $s_1$ has no neighbours in $(N(v_1)\cap N(v_2))-N(v_3)$  (and by symmetry no neighbours in $(N(v_3)\cap N(v_2))-N(v_1)$). Suppose $n\in (N(v_1)\cap N(v_2))-N(v_3)$ such that $s_1\sim n$. Now, $\{s_1,v_1,v_2,v_3,n\}$ induces a $bull$ in $G$, a contradiction. 

Since $s_1$ is not comparable with $v_1$, then it must have a neighbour $u_1$ such that $u_1\nsim v_1$. By Lemma~\ref{lem:2P2freeneighboursofS}, $u_1\not\in \overline{N(v_2)}$ (or else $u_1\sim v_1$) and therefore $u_1\sim v_2$. Now, since $s_1$ has no neighbours in $(N(v_3)\cap N(v_2))-N(v_1)$ as argued above, it must be that $u_1\nsim v_3$. 

If $v_1$ and $v_3$ are not comparable, then they must have distinct neighbours. Let $v_1'$ be such that $v_1\sim v_1'$ and $v_1'\nsim v_3$ and $v_3'$ be such that $v_3\sim v_3'$ and $v_3'\nsim v_1$. Note that we must have $v_1'\sim v_3'$, otherwise $\{v_1,v_1',v_3,v_3'\}$ induces a $2P_2$ in $G$.
Further $u_1$ must be complete to $\{v_1',v_3'\}$, else $\{s_1,u_1,v_1,v_1'\}$ or $\{s_1,u_1,v_3,v_3'\}$ will induce a $2P_2$ in $G$. But now, $\{u_1,v_1',v_3',v_1,v_3\}$ induces a $bull$ in $G$, a contradiction. Figure~\ref{fig:bullproofillustration} illustrates the induced $bull$ in $G$	q, although the reader should note that while all adjacencies and nonadjacencies are complete within the set $\{u_1,v_1',v_3',v_1,v_3\}$, there could be other missing adjacencies in the figure (for example, we might have $v_1'\sim v_2$). Thus, $v_1$ and $v_3$ are comparable, contradicting $G$ being $k$-vertex-critical.



\end{proof}

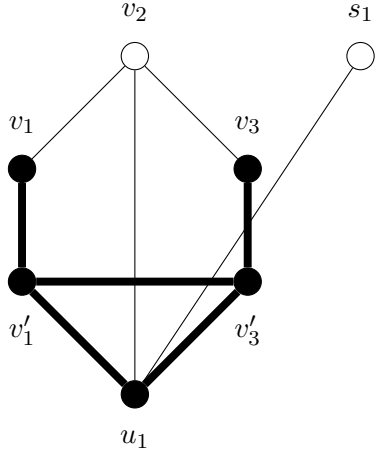
\begin{figure}[h]
\def\r{1.5}
\centering
\begin{tikzpicture}
\begin{scope}[every node/.style={circle,draw}]
   \node[label=above:$v_2$,draw] (v2) at (0*\r,1*\r) {};
   \node[label=above:$s_{1}$,draw] (s1) at (2*\r,1*\r) {};    
\end{scope}
\begin{scope}[every node/.style={circle,fill,draw}]
	\node[label=above:$v_1$,draw] (v1) at (-1*\r,0*\r) {};
 	\node[label=above:$v_3$,draw] (v3) at (1*\r,0*\r) {};
	
     \node[label=below:$u_1$,draw] (u1) at (0*\r,-2*\r) {};  
     \node[label=below:$v_1'$,draw] (v1p) at (-1*\r,-1*\r) {};
     \node[label=below:$v_3'$,draw] (v3p) at (1*\r,-1*\r) {};    
\end{scope}

\begin{scope}[line width=3pt]
	\path [-] (v1) edge node {} (v1p);
	\path [-] (v3) edge node {} (v3p);
	\path [-] (v1p) edge node {} (v3p);
    \path [-] (u1) edge node {} (v3p);
	\path [-] (u1) edge node {} (v1p);

\end{scope}

\begin{scope}
    \path [-] (v1) edge node {} (v2);
	\path [-] (u1) edge node {} (s1);
	\path [-] (v2) edge node {} (v3);
	\path [-] (u1) edge node {} (v2);
  
\end{scope}


 \end{tikzpicture}
\caption{An illustration of part of the proof of Lemma~\ref{lem:2P2bullP3P1fee} with the induced $bull$ in bold.}\label{fig:bullproofillustration}
\end{figure}

The following theorem follows directly from Lemma~\ref{lem:2P2bullP3P1fee} and Theorem~\ref{thm:finiteP3ellP1freecrit}.

\begin{theorem}
There are only finitely many $k$-vertex-critical $(2P_2,bull)$-free graphs for all $k\ge 1$.
\end{theorem}

\section{($2P_2$, \squid{\ell})-free}\label{sec:ellsquid}
Recall that the \squid{\ell}  is the graph obtained from a $C_4$ by adding $\ell$ leaves to one vertex, see Figure~\ref{subfig:4squid}.
%
%
%
%
%
%
%
%
%
%
%

\begin{lemma}\label{lem:2P2ellsquidP3cP1fee}
Let $\ell,k\ge 1$ and $c=(\ell-1)(k-1)+1$. If $G$ is a $k$-vertex-critical $(2P_2$, \squid{\ell}$)$-free graph, then $G$ is $(P_3+ cP_1)$-free.
\end{lemma}
\begin{proof}
Let $G$ be a $k$-vertex-critical ($2P_2$, \squid{\ell})-free graph for some $k,\ell\ge 1$ and let $c=(\ell-1)(k-1)+1$. Suppose by way of contradiction that $G$ contains an induced $P_3+c P_1$ with $\{v_1,v_2,v_3\}$ inducing the $P_3$ in that order and $S=\{s_1,s_2\dots, s_c\}$ the $c P_1$ of the induced $P_3+c P_1$. 

By Lemma~\ref{lem:2P2freenonneighbconnected}, each $s_i$ must have a neighbour in $\overline{N[v_2]}$. Further, by Lemma~\ref{lem:2P2freeneighboursofS}, for every $u\in \overline{N[v_2]}-S$, such that $u\sim s_i$ for some $s_i\in S$, we must have that $u$ is complete to $\{v_1,v_3\}$. Therefore, each $u\in \overline{N[v_2]}-S$  has at most $\ell-1$ neighbours in $S$, else $u,v_1,v_2,v_3$ together with any $\ell$ of $u$'s neighbours in $S$ would induce an \squid{\ell}. Let $U=\{u_1,u_2,\dots u_m\}$ be a subset of $N(S)\cap \overline{N[v_2]}$ such that $\left(\bigcup_{i=1}^{m} N(u_i)\right)\cap S = S$ and such that $N(u_i)\cap S\not\subseteq N(u_j)\cap S$ for all $i\neq j$. Such a set $U$ exists since each vertex in $S$ has at least one neighbour in $\overline{N[v_2]}$. 
Since each $u_i$ can have at most $\ell-1$ neighbours in $S$ we must have that $|U|\ge k$ by the Pigeonhole Principle. Let $u_i,u_j\in U$ for $i\neq j$ and without loss of generality let $s_i\in N(u_i)-N(u_j)$ and $s_j\in N(u_j)-N(u_i)$. If $u_i\nsim u_j$, then $\{s_i,u_i,s_j,u_j\}$ induces a $2P_2$ in $G$, a contradiction. Therefore, $u_i\sim u_j$ for all $i\neq j$ and therefore $U$ induces a clique with at least $k$ vertices in $G$. But now $U$ induces a proper subgraph of $G$ that requires at least $k$ colours, which contradicts $G$ being $k$-vertex-critical. Therefore, $G$ must be $(P_3+ cP_1)$-free.
\end{proof}

The following theorem follows directly from Lemma~\ref{lem:2P2ellsquidP3cP1fee} and Theorem~\ref{thm:finiteP3ellP1freecrit}.

\begin{theorem}\label{thm:2P2ellsquidcrit}
There are only finitely many $k$-vertex-critical $(2P_2$, \squid{\ell}$)$-free graphs for all $k,\ell\ge 1$.
\end{theorem}

Since $chair$ is an induced subgraph of \squid{2}, $claw+P_1$ is an induced subgraph of \squid{3}, and more generally  $K_{1,\ell}+P_1$ is an induced subgraph of \squid{\ell} for all $\ell\ge 1$, we get the following immediate corollaries of Theorem~\ref{thm:2P2ellsquidcrit} 

\begin{corollary}
There are only finitely many $k$-vertex-critical $(2P_2, chair)$-free graphs for all $k\ge 1$.
\end{corollary}

\begin{corollary}
There are only finitely many $k$-vertex-critical $(2P_2, claw+P_1)$-free graphs for all $k\ge 1$.
\end{corollary}

\begin{corollary}
There are only finitely many $k$-vertex-critical $(2P_2,K_{1,\ell}+P_1)$-free graphs for all $k\ge 1$.
\end{corollary}

We note as well that $banner$ is another name for  \squid{1}, so our results also imply that there are only finitely many $k$-vertex-critical ($2P_2$, $banner$)-free graphs for all $k$. This result is not new though; it was recently shown in~\cite{Brause2022} that every $k$-vertex-critical ($P_5$, $banner$)-free graph has independence number less than $3$ and therefore there only finitely many such graphs by Ramsey's Theorem.  However, a special case of Lemma~\ref{lem:2P2ellsquidP3cP1fee} implies that every $k$-vertex-critical ($2P_2$, $banner$)-free graph is $(P_3+P_1)$-free and therefore by Theorem 3.1 in~\cite{CameronHoangSawada2022}, every such graph has independence number less than $3$. Thus our results give a new and extremely short proof (in fact the proof of Lemma~\ref{lem:2P2ellsquidP3cP1fee} only requires the first four sentences if restricted to $banner$-free graphs) of a slightly weaker version of the result in~\cite{Brause2022}.

\section{($2P_2$, \hl{\ell})-free}\label{sec:fountain}

Recall that \hl{\ell}, shown in Figure~\ref{subfig:3squid}, is the graph obtained from a triangle by attaching $\ell$ leaves to one of its vertices.

\begin{lemma}\label{lem:2P2Hellfreeimpliesellsquidfree}
Let $\ell,k\ge 1$. If $G$ is $k$-vertex-critical $(2P_2$, \hl{\ell}$)$-free, then $G$ is \squid{2\ell-1}-free.
\end{lemma}
\begin{proof}
Let $G$ be a $k$-vertex-critical ($2P_2$, \hl{\ell})-free graph and suppose by way of contradiction that $G$ contains an induced \squid{2\ell-1} with the same labelling as Figure~\ref{subfig:4squid}. It must be the case that there are $u_2',u_3'\in V(G)$ such that  $u_2'\sim u_2$, $u_2'\nsim u_3$, $u_3'\sim u_3$, and $u_3'\nsim u_2$ otherwise $u_2$ and $u_3$ are comparable, contradicting Lemma~\ref{lem:nocomparablecliques}. Now, $\{ u_2, u_3, u'_2, u'_3 \}$ induces a $2P_2$ unless $u'_2 \sim u'_3$. There are now two cases to consider.\\

\noindent \textit{Case 1:} $u_2'\sim u_4$.

In this case, $u_2'$ must be adjacent to at least $\ell$ of the $w_i$'s, or else $\{u_2',u_2,u_4\}$ together with any $\ell$ of the $w_i$'s that are nonadjacent to $u_2'$ induce a \hl{\ell}, a contradiction. Let $W$ be the subset of the $w_i$'s that are adjacent to $u_2'$. We now have $\{u_2',w,u_1,u_3\}$ inducing a $2P_2$ for all $w\in W$ unless $u_2'\sim u_1$. Since $G$ is $2P_2$-free, we must have $u_2'\sim u_1$. But now $\{u_2,u_2',u_1\}$ together with any subset of $W$ with at least $\ell$ vertices induces a \hl{\ell}, a contradiction.\\

 \noindent \textit{Case 2:} $u_2'\nsim u_4$.

In this case, $u_2'\sim w_i$ for all $i\in\{1,\dots,\ell\}$, else $\{u_2,u_2',u_4,w_i\}$ would induce a $2P_2$. We now have $u_2'\sim u_1$, else $\{u_2',w_i,u_1,u_3\}$ induces a $2P_2$ for any $i\in\{1,\dots,\ell\}$. But now $\{u_2,u_2',u_1\}\cup\{w_1,w_2,\dots, w_{\ell}\}$ induces \hl{\ell}, a contradiction.\\

Since we reach contradictions in either case, we must contradict that fact that $G$ is \squid{2\ell-1}-free.
\end{proof}

Thus, we obtain the following theorem directly from Lemma~\ref{lem:2P2Hellfreeimpliesellsquidfree} and Theorem~\ref{thm:2P2ellsquidcrit}.

\begin{theorem}
There are only finitely many $k$-vertex-critical $($\hl{\ell}$)$-free graphs for all $k,\ell\ge 1$.
\end{theorem}

Since the graph $\overline{diamond+P_1}$ is another name for to \hl{2}, we get the following immediate corollary.

\begin{corollary}
There are only finitely many $k$-vertex-critical $(\overline{diamond+P_1})$-free graphs for all $k\ge 1$.
\end{corollary}
%
%

\section{Exhaustive Generation for Small $k$}\label{sec:critgen}

Our results imply that there are several new families of graphs for which \kcol{} can be solved in polynomial-time and certified with a $k$-colouring or a $(k+1)$-vertex-critical induced subgraph graph for any fixed $k$. The implementation of any of these algorithms, however, requires the complete set of all $(k+1)$-vertex-critical graphs in that family. For two of the families, we can give these complete sets for all $k\le 6$, and therefore the corresponding polynomial-time certifying \kcol{} algorithms can be readily implemented. These two families are $(2P_2,bull)$-free and $(2P_2,banner)$-free since $k$-vertex-critical graphs in both families were shown to be $(P_3+P_1)$-free.
In~\cite{CameronHoangSawada2022} it was shown that every $k$-vertex-critical $(P_3+P_1)$-free graph has independence number at most $2$ and order at most $2k-1$, which allowed for the exhaustive generation of all such graphs for $k\le 7$. These graphs, available at~\cite{graphfiles}, were searched for those that are also $2P_2$-free to provide the complete sets of all $k$-vertex-critical $(2P_2,H)$-free graphs for $H=banner$ or $H=bull$ for all $k\le 7$.  Note that since $\alpha(banner)=\alpha(bull)=3$, $banner$ and $bull$ are already forbidden induced subgraphs from each of the $k$-vertex-critical $(P_3+P_1)$-free graphs since they necessarily have independence number at most $2$. The graphs in graph6 format are available at~\cite{2P2bullfreefiles} and the number of each order is summarized in Table~\ref{tab:5to72P2bullcrit}.

\begin{table}[!h]
\begin{center}  \small
\renewcommand\arraystretch{1.2}
\begin{tabular}{|r|r|r|r|r|} 
\hline
$n$  &  $4$-vertex-critical  & $5$-vertex-critical & $6$-vertex-critical & $7$-vertex-critical  \\ \hline
4   	& $1$  & $0$      &  $0$     & $0$ \\ \hline
5   	& $0$  & $1$      &  $0$     & $0$ \\ \hline
6  		& $1$  & $0$   	  &  $1$     & $0$ \\ \hline
7   	& $2$  & $1$   	  &  $0$     & $1$ \\ \hline
8   	& $0$  & $2$   	  &  $1$     & $0$ \\ \hline
9   	& $0$  & $11$ 	  &  $2$     & $1$  \\ \hline
10  	& $0$  & $0$   	  &  $12$    & $2$    \\ \hline
11  	& $0$  & $0$   	  &  $126$   & $12$ \\ \hline 
12  	& $0$  & $0$   	  &  $0$  	 & $128$ \\ \hline
13  	& $0$  & $0$   	  &  $0$  	 & $3806$ \\ \hline \hline
total 	& $4$  & $15$ 	  & $142$ 	 & $3947$ \\ \hline
\end{tabular}
\caption{Number of $k$-vertex-critical $(2P_2,H)$-free graphs of order $n$ for $k\le 7$ where $H$ is $banner$ or $bull$.}\label{tab:5to72P2bullcrit}
\end{center}
\end{table}

\section{Conclusion}\label{sec:conclusion}

Of the 11 graphs $H$ of order $5$ where it remains unknown if there are only finitely many $k$-vertex-critical $(P_5,H)$-free graphs for all $k$, we have shown finiteness for four of them in the more restricted $2P_2$-free case. A clear open question that remains is if these still hold when forbidding $P_5$ instead of $2P_2$, or (perhaps even more interestingly) if for some they do not. The question also remains completely open for the other $7$ graphs of order $5$; the one we find most interesting is $H=P_4+P_1$, as it is still unknown whether there are only finitely many $k$-vertex-critical $(P_4+P_1)$-free for any given $k\ge 5$ even without additional restrictions.  We fully expect that our new proof technique of reducing to $(P_3+\ell P_1)$-free for some $\ell$ will be useful for proving the finiteness of $k$-vertex-critical graphs in these open cases and indeed for other families of graphs. One place to start for interested readers might be by forbidding additional induced subgraphs and, in particular,  considering $k$-vertex-critical $(2P_2, K_3+P_1,H)$-free or even $(2P_2, K_3+P_1,C_5,H)$-free graphs. These additional forbidden induced subgraphs may seem to be far too restrictive, but we feel they are justified entry points as all known infinite families of $k$-vertex-critical graphs that are $2P_2$-free are also $(K_3+P_1)$-free and, for $k\ge 6$, there are families of $k$-vertex-critical $(2P_2,K_3+P_1,C_5)$-free graphs (see the aforementioned construction in~\cite{CameronHoang2023P5C5}).

\section*{Acknowledgements}
The first two authors were supported by the Natural Sciences and Engineering Research Council of Canada (NSERC)  USRA program. The third author also gratefully acknowledges research support from NSERC (grants RGPIN-2022-03697 and DGECR-2022-00446) and Alberta Innovates. The research support from Alberta Innovates was used by the third author to hire the fourth author.

%
%

\bibliographystyle{abbrv}
\bibliography{refs}

\end{document}